\RequirePackage{fix-cm}
\documentclass[12pt,english]{article}
\usepackage{amsmath}
\usepackage[T1]{fontenc}
\usepackage[letterpaper]{geometry}
\usepackage{amssymb}
\usepackage{comment}
\usepackage{latexsym}

\usepackage{graphicx}

\usepackage[backref,colorlinks=true]{hyper ref}
\usepackage{enumerate}
\usepackage[dvipsnames,usenames]{color}
\newcommand{\mc}[1]{\ensuremath{\mathcal{#1}}} 
\newcommand{\EE}{{\mathbb E}}

\definecolor{highlight}{rgb}{1,0.2,0.2}

\newtheorem{open}{Open Question}

\def\BBox{\kern  -0.2cm\hbox{\vrule width 0.2cm height 0.2cm}}

\newtheorem{theorem}{Theorem}[section]

\begin{document}

\title{Prescribing Symmetries and Automorphisms for Polytopes}
\author{
Egon Schulte\thanks{Partially supported by Simons Foundation award no. 420718.}
\thanks{Email: e.schulte@northeastern.edu},\\
Northeastern University,\\
Department of Mathematics,\\
Boston, MA 02115, USA
\\[.25in]
Pablo Sober\'{o}n\thanks{Partially supported by NSF Grant DMS 1851420} \thanks{Email: pablo.soberon-bravo@baruch.cuny.edu},\\
Baruch College, City University of New York,\\
Department of Mathematics,\\
New York, New York 10010, USA
\\[.25in]
Gordon Ian Williams
\thanks{Email:\ giwilliams@alaska.edu}\\
University of Alaska Fairbanks\\
Department of Mathematics and Statistics,\\
Fairbanks, AK 99709, USA}

\maketitle

\begin{abstract}
We study finite groups that occur as combinatorial automorphism groups or geometric symmetry groups of convex polytopes. When $\Gamma$ is a subgroup of the combinatorial automorphism group of a convex $d$-polytope, $d\geq 3$, then there exists a convex $d$-polytope related to the original polytope with combinatorial automorphism group {\it exactly\/} $\Gamma$. When $\Gamma$ is a subgroup of the geometric symmetry group of a convex $d$-polytope, $d\geq 3$, then there exists a convex $d$-polytope related to the original polytope with both geometric symmetry group and combinatorial automorphism group {\it exactly\/} $\Gamma$. These symmetry-breaking results then are applied to show that for every abelian group $\Gamma$ of even order and every involution $\sigma$ of $\Gamma$, there is a centrally symmetric convex polytope with geometric symmetry group $\Gamma$ such that $\sigma$ corresponds to the central symmetry.
\end{abstract}


\section{Introduction}\label{intro}

The study of convex polytopes is largely motivated by their symmetries.  With every convex polytope $P$ are associated two finite groups:\ the (geometric) symmetry group $G(P)$ consisting of the Euclidean isometries of the ambient space that preserve $P$, and the (combinatorial) automorphism group $\Gamma(P)$ consisting of the combinatorial symmetries of the face lattice of~$P$. It is natural to ask about whether or not the converse is true:\ is every finite group the symmetry group or automorphism group of a convex polytope?

For automorphism groups this question was answered positively by Schulte and Williams \cite{SWil}, and later a simpler proof was found by Doignon \cite{Do1}. In this paper we are studying variations of this question with additional restrictions imposed on the polytopes in question. We are particularly interested in \textit{centrally symmetric\/} convex polytopes in Euclidean $d$-space~$\mathbb{E}^d$. By definition these admit the reflection in the origin, $x\mapsto -x$, as a geometric symmetry and thus have an automorphism group (as well as symmetry group) that contains an  involution. The main motivation for this paper was to characterize the pairs $(\Gamma,\sigma)$, consisting of a finite group $\Gamma$ and an involution $\sigma$ in $\Gamma$, with the property that $\Gamma$ is the automorphism group of a \textit{centrally symmetric} convex polytope such that $\sigma$ corresponds to the central symmetry. In Theorem~\ref{thm3} we show that every abelian group $\Gamma$ of even order has the desired property:\ for every involution $\sigma$ in $\Gamma$ there is a centrally symmetric polytope with automorphism group $\Gamma$ such that $\sigma$ acts like the central symmetry.

Along the way we generalize the methods of \cite{SWil} to establish the following two symmetry-breaking results for arbitrary convex polytopes, which are applicable in a wider context and are of independent interest. When $\Gamma$ is a subgroup of the automorphism group of some convex $d$-polytope $Q$, $d\geq 3$, then there exists a convex $d$-polytope $P$ related to~$Q$ with automorphism group {\it exactly\/} $\Gamma$. When $\Gamma$ is a subgroup of the geometric symmetry group of some convex $d$-polytope~$Q$, $d\geq 3$, then there exists a convex $d$-polytope $P$ related to~$Q$ with both symmetry group and automorphism group {\it exactly\/} $\Gamma$. Our symmetry-breaking constructions are described in Section~\ref{section-preassigning} and generalize to some extent to abstract polytopes (see \cite{McMSch02}).  In Section \ref{section-involutions} we investigate centrally symmetric polytopes. Finally, in Section~\ref{section-openprobs} we discuss some open problems and point to recent solutions.

The question of finding polytopes with prescribed automorphism group has also been asked as motivated by representation theory, see \cite{Lad16, BL18, FL18}.  These articles study orbit polytopes, that is, convex hulls of single  point orbits under finite groups acting affinely on a real vector space. In this context it is natural to additionally consider the ``affine symmetry group'' (sometimes also called the ``affine automorphism group'') of a convex polytope, consisting of all non-singular affine transformations of the ambient space that preserve the polytope. As not every finite group is the affine automorphism group of an orbit polytope, it seems that symmetry-breaking processes as described here cannot be completely avoided to settle the above problem.  

The question whether or not a given group is the automorphism group or symmetry group of a geometric, combinatorial, algebraic, or topological structure of a specified kind has been studied quite extensively. For a recent article describing the common characteristics of the approaches see the recent article \cite{Jon18} by Jones.

\section{Basic Notions}

We begin by recalling some basic definitions from the theory of convex and abstract polytopes (see~\cite{Gru03,McMSch02,Zie95}).

An {\it abstract polytope\/} of rank $d$ is a ranked poset $\mathcal{P}$ with the following properties. The elements of $\mathcal{P}$ are called {\em faces\/}, and the possible face ranks are $-1,0,\ldots,d$. A face is a $j$-{\it face} if its rank is $j$. Faces of ranks 0, 1 or $d-1$ are also called {\it vertices}, {\it edges} or {\it facets} of $\mathcal{P}$, respectively. The poset $\mathcal{P}$ has a smallest face (of rank $-1$) denoted $F_{-1}$ and a largest face (of rank $d$) denoted $F_d$. Each {\it flag} (maximal totally ordered subset) $\Phi$ of $\mathcal{P}$ contains exactly $d+2$ faces, one for each rank~$j$. Two flags are said to be {\it adjacent} if they differ in just one face; they are {\em $j$-adjacent\/} if this face has rank $j$. The poset $\mathcal{P}$ is {\it strongly flag-connected}, meaning that any two flags $\Phi$ and $\Psi$ can be joined by a sequence of flags $\Phi=\Phi_0,\Phi_1,...,\Phi_k=\Psi$, all containing $\Phi\cap\Psi$, such that any two successive flags $\Phi_{i-1}$ and $\Phi_i$ are adjacent. Finally, $\mathcal{P}$ satisfies the {\it diamond condition}:\  whenever $F\leq G$, with ${\rm rank}(F)=j-1$ and ${\rm rank}(G)=j+1$, there are exactly two faces $H$ of rank $j$ such that $F\leq H\leq G$. Thus, for $j=0,\ldots,d-1$, a flag of $\mathcal{P}$ has exactly one $j$-adjacent flag. 

If $F$ and $G$ are faces with $F \leq G$, then $G/F := \{H \mid F \leq H \leq G\}$ is called a {\em section\/} of~$\mathcal{P}$. This is a polytope in its own right. For a face $F$, we also call $F_{d}/F$ the {\em co-face\/} of $\mathcal{P}$ at~$F$, or the {\em vertex-figure\/} of $\mathcal{P}$ at $F$ if $F$ is a vertex.  

The face lattice of a convex polytope is an example of an abstract polytope. Recall that a {\em convex polytope} $P$ is the convex hull of finitely many points in Euclidean $d$-space $\mathbb{E}^{d}$. A (proper) {\em face\/} of a convex $d$-polytope $P$ is the intersection of $P$ with a supporting hyperplane of $P$; the latter is a hyperplane $H$ in $\mathbb{E}^d$ such that $P$ lies entirely in one of the two closed half-spaces bounded by $H$ and has points in common with $H$. The empty set $\emptyset$, and $P$ itself, are also called (improper) faces of $P$. The set of all (proper and improper) faces of a convex polytope $P$, ordered by inclusion, forms a lattice called the {\em face lattice\/} of $P$. This is an abstract polytope, of rank $d$ if $P$ has dimension $d$. The \emph{boundary complex} of a convex $d$-polytope $P$, denoted ${\rm bd}(P)$, is the set of all faces of $P$ of rank less than $d$, partially ordered by inclusion (see~\cite[p. 40]{Gru03}); this complex tessellates the boundary $\partial P$ of $P$ and is topologically a $(d-1)$-sphere. 

Recall that a convex $d$-polytope is called {\it simple\/} if all its vertices have valency $d$, and {\it simplicial\/} if all its facets are $(d-1)$-simplices.

Let $P$ be a convex $d$-polytope. The (standard) \emph{barycentric subdivision} of $P$ is the geometric simplicial complex of dimension $d$, whose $d$-simplices are precisely the convex hulls of the centroids of the non-empty faces in a flag of $P$ (see \cite{Bay88}, \cite[p.~642]{GooORo04} or \cite[Sect. 2C]{McMSch02}). We use the term ``barycentric subdivision'' more broadly and allow the centroid of a face to be replaced by a relative interior point of that face. Thus, a barycentric subdivision of $P$ is a $d$-dimensional geometric simplicial complex with one vertex in the relative interior of each non-empty face of $P$, and with one $d$-dimensional simplex per flag of $P$, such that the vertices of a $d$-simplex are precisely the relative interior points chosen in the faces of the corresponding flag. Each barycentric subdivision of $P$ is isomorphic (as an abstract simplicial complex) to the order complex of the face lattice of $P$ (with the empty face removed); in particular, any two barycentric subdivisions are isomorphic.

There is a similar notion of barycentric subdivision for the boundary complex of a convex polytope. By $\mathcal{C}(P)$ we denote the barycentric subdivision of the boundary complex ${\rm bd}(P)$ of~$P$. This is a $(d-1)$-dimensional simplicial complex.

The order complex of an abstract polytope $\mathcal{P}$ similarly can be viewed as a ``combinatorial barycentric subdivision'' of $\mathcal{P}$ (see~\cite[Sect. 2C]{McMSch02}).

The {\it $k$-skeleton\/} ${\rm skel}_{k}(\mathcal{P})$ of an abstract polytope $\mathcal{P}$ is the poset consisting of all proper faces of $\mathcal{P}$ of rank at most $k$ (together with the induced partial order).

\section{Preassigning symmetry groups}\label{section-preassigning}

We begin this section with the following theorem about symmetry-breaking in convex polytopes.

\begin{theorem}
\label{thm1}
Let $d\geq 3$, let $Q$ be a convex $d$-polytope with (combinatorial) automorphism group $\Gamma(Q)$, and let $\Gamma$ be a subgroup of $\Gamma(Q)$. Then there exists a finite abstract $d$-polytope~$\mathcal{P}$ with the following properties:\\
(a)\ $\Gamma(\mathcal{P})=\Gamma$.\\
(b)\ $\mathcal{P}$ is isomorphic to a face-to-face tessellation $\mathcal{T}$ of the $(d-1)$-sphere $\mathbb{S}^{d-1}$ by spherical convex $(d-1)$-polytopes.\\
(c)\ ${\rm skel}_{d-2}(\mathcal{C}(Q))$ is a subcomplex of ${\rm skel}_{d-2}(\mathcal{P})$.\\
(d)\ If $\Gamma$ is a subgroup of the (geometric) symmetry group $G(Q)$ of $Q$, then the tessellation $\mathcal{T}$ on $\mathbb{S}^{d-1}$ in part (b) can be chosen in such a way that $G(\mathcal{T})=\Gamma=\Gamma(\mathcal{T})$.
\end{theorem} 

First notice that the conclusion of the theorem above may fail for $d=2$, since the combinatorial automorphism group of a finite abstract $2$-polytope (polygon) is necessarily dihedral and so in particular cannot be cyclic. Hence we must require $d\geq 3$.
\bigskip

\begin{proof} 
We begin with the second and third parts of the theorem, then settle the first part, and later refine our arguments to settle the fourth part. Our strategy is to refine the structure of the given convex polytope $Q$ in such a way that all automorphisms in $\Gamma(Q)$ outside $\Gamma$ are destroyed. The result will be a spherical abstract polytope whose automorphism group is given by $\Gamma$. This symmetry-breaking process is interesting in its own right.
\smallskip

{\sf Parts~(b,c).}\ Consider the (standard or any other) barycentric subdivision $\mathcal{C}(Q)$ of the boundary complex ${\rm bd}(Q)$ of $Q$ in d-space $\EE^d$. This is a simplicial $(d-1)$-complex that refines ${\rm bd}(Q)$ and is a realization of the order complex of ${\rm bd}(Q)$ (see~\cite[Sect. 2C]{McMSch02}). Its simplices correspond to the chains (totally ordered subsets) in the poset ${\rm bd}(Q)$, with the chambers (maximal simplices) corresponding to the flags of 
${\rm bd}(Q)$; here, inclusion of simplex faces in $\mathcal{C}(Q)$ corresponds to inclusion of chains in ${\rm bd}(Q)$. In particular, $\mathcal{C}(Q)$ has the structure of a {\em labelled\/} simplicial complex, in which every simplex is labelled by the set of ranks of the faces in the chain of ${\rm bd}(Q)$ represented by the simplex. Thus the vertices of $\mathcal{C}(Q)$ can be labelled by~$0,\ldots,d-1$. The vertices of $Q$ are exactly the vertices of $\mathcal{C}(Q)$ with label $0$. The vertices of each chamber are labelled $0,\ldots,d-1$ such that no two vertices have the same label. Note that $\Gamma(Q)$ and hence $\Gamma$ act on $\mathcal{C}(Q)$ as groups of automorphisms of a labelled simplicial complex (labels of simplices are preserved), and that the action on the chambers is free.

As in the proof of \cite[Theorem~1]{SWil}, a key step in the construction consists of chamber replacement by complexes made up of Schlegel diagrams of convex polytopes. These complexes are inserted into the chambers of $\mathcal{C}(Q)$ in such a way that the $(d-2)$-skeleton ${\rm skel}_{d-2}(\mathcal{C}(Q))$ of $\mathcal{C}(Q)$ stays intact, unrefined. In fact, our proof basically consists of adapting the proof of \cite[Theorem~1]{SWil} to the more general situation at hand. (In the proof of that theorem, the corresponding subgroup $\Gamma$ acted simply vertex-transitively on a special convex $d$-polytope $Q$ constructed from a suitable permutation representation of $\Gamma$. In the present context, $Q$ can be an arbitrary $d$-polytope and $\Gamma$ need not act vertex-transitively.)

The complexes inserted into the chambers are constructed in exactly the same manner as in \cite{SWil}. We will review the properties of these complexes below. Each complex is built from a Schlegel diagram $\mathcal{D}$ of a $d$-crosspolytope supported on a $(d-1)$-dimensional simplex $D$ with vertices $u_{0},\ldots,u_{d-1}$, by inserting affine images of Schlegel diagrams of certain convex $d$-polytopes (the polytopes $R_i$ and $L$ described below) into the $(d-1)$-simplices of $\mathcal{D}$ that correspond to certain facets of the $d$-crosspolytope. The resulting $(d-1)$-dimensional complex, which as in \cite{SWil} is denoted $\mathcal{R}^L$, is also supported on $D$ and has the boundary complex of $D$ as a subcomplex. The particular choice of the polytopes $R_i$ and $L$ is quite delicate and is taken in such a way that the vertices $u_{0},\ldots,u_{d-1}$ of the outer simplex $D$ acquire very high valencies in $\mathcal{R}^L$ compared with the vertices in the interior of $D$, and that the valencies of these vertices in $\mathcal{R}^L$ are integers ``very far apart'' from each other. These conditions on the insertion process later prevent the existence of unwanted automorphisms. In particular, $\mathcal{R}^L$~itself will have no automorphism other than the trivial automorphism.

Before moving on to the actual chamber insertion process we briefly review the construction and properties of the complexes $\mathcal{R}^L$. First recall that the Schlegel diagram $\mathcal{D}$ of a $d$-crosspolytope consists of an {\em outer\/} $(d-1)$-simplex $D$, tiled in a face-to-face manner by $(d-1)$-simplices, the {\em simplex tiles\/} of~$\mathcal{D}$. Among these simplex tiles is a {\em central\/} $(d-1)$-simplex $Z$, corresponding to the facet of the crosspolytope opposite to the facet defining $D$. The simplices $D$ and $Z$ have no vertices in common. The simplex tiles of $\mathcal{D}$ adjacent to $Z$ (i.e., intersecting $Z$ in a common facet) share precisely one vertex with $D$; conversely, every vertex $u$ of $D$ is a vertex of precisely one simplex tile, $F_u$ (say), that is adjacent to $Z$. In the course of the construction we often require affine images of Schlegel diagrams supported on $(d-1)$-simplices. Clearly, any affine transformation that carries the supporting $(d-1)$-simplex of a Schlegel-diagram to another $(d-1)$-simplex, also carries the Schlegel diagram on the first simplex to a ``diagram'' on the second simplex (this also is a Schlegel diagram of some polytope). This is true no matter how the vertices of the first $(d-1)$-simplex are assigned by the affine transformation to the vertices of the second. 

The next step is to modify $\mathcal{D}$ in such a way that the vertices in the outer simplex $D$ acquire very high valencies compared with those in the interior, and that the valencies of the vertices of $D$ are very far apart from each other. To this end, consider the simplex tiles $F_{u_0},\ldots,F_{u_{d-1}}$ of $\mathcal{D}$ determined by the vertices $u_{0},\ldots,u_{d-1}$ of $D$, and replace every simplex tile $F_{u_i}$ by an affine image of the Schlegel diagram of a suitable convex $d$-polytope $R_i$. All vertices of this polytope $R_i$, save one, have small valencies but the exceptional vertex (which is mapped to $u_i$) has valency given by a large integer $m_i$ to be determined. For example, for~$R_i$ we could take the pyramid over a {\it simple\/} convex $(d-1)$-polytope with $m_i$ vertices and with at least one facet which is a simplex; then $R_i$ itself has a simplex facet, with the apex of $R_i$ as a vertex of valency $m_i$ in $R_i$. Suppose $R_i$ is a pyramid of this kind. Then $R_i$ admits a Schlegel diagram $\mathcal{R}_i$ whose outer $(d-1)$-simplex corresponds to a simplex facet of $R_i$ containing the apex of $R_i$. In this diagram, the outer vertex representing the apex has valency $m_i$ while all other vertices have (small) valency $d$. Now take an affine transformation that maps the outer simplex of $\mathcal{R}_i$ to the simplex tile $F_{u_i}$ of $\mathcal{D}$ such that  the vertex corresponding to the apex is mapped to $u_i$, and then insert the corresponding affine image of the Schlegel diagram $\mathcal{R}_i$ into the simplex $F_{u_i}$ such that $F_{u_i}$ becomes the outer simplex. If this procedure is performed for each $i=0,\ldots,d-1$, the result is a $(d-1)$-dimensional complex $\mathcal{R}$ supported on $D$, in which each vertex $u_i$ of $D$ has large valency, namely $m_{i}+d-1$, while all vertices of $\mathcal{R}$ that are not vertices of $D$ have small valencies. 

We require one additional type of modification to complete the construction of $\mathcal{R}^L$, now targeting the $(d-1)$-simplex of $\mathcal{R}$ that was the central simplex of $\mathcal{D}$. Suppose $L$ is any {\it simplicial} convex $d$-polytope. Then we let $\mathcal{R}^L$ denote the $(d-1)$-dimensional complex supported on $D$, in which the central simplex has been replaced by a suitable affine copy of a Schlegel diagram of $L$.

At this point of the construction we still have the choice of the parameters $m_0,\ldots,m_{d-1}$ and the polytopes $L$ at our disposal. These will be chosen as we move along and will depend on the given polytope $Q$. 

The chamber insertion process for the polytope $Q$ employs the action of $\Gamma$ as a group of label preserving automorphisms on the barycentric subdivision $\mathcal{C}(Q)$. For a chamber $C$ of~$\mathcal{C}(Q)$, we let $o(C)$ denote the orbit of $C$ under $\Gamma$ in its action on $\mathcal{C}(Q)$. 

The chamber replacement now proceeds as follows. We first settle the choice of the polytopes $L$. For each chamber orbit $o(C)$ choose a simplicial convex $d$-polytope $L_{o(C)}$ in such a way that no two such polytopes have the same number of vertices. Then, for any fixed choice of parameters $m_0,\ldots,m_{d-1}$ (and associated complex $\mathcal{R}$), no two of the corresponding $(d-1)$-dimensional complexes $\mathcal{R}^{L_{o(C)}}$ have the same number of vertices, and thus no two complexes are combinatorially isomorphic. 

In the final step of the chamber insertion process we first replace, for each chamber orbit $o(C)$, one of its chambers, $C$ (say), by an affine copy of the corresponding complex $\mathcal{R}^{L_{o(C)}}$ such that, for each $i=0,\ldots,d-1$, the vertex $u_i$ of $D$ is mapped onto the vertex of $C$ labelled $i$ in $\mathcal{C}(Q)$. We then exploit $\Gamma$ to carry this new structure to all the other chambers in an orbit, and therefore to all chambers of $\mathcal{C}(Q)$. Recall that $\Gamma(Q)$, and hence $\Gamma$, acts freely and in a label preserving manner on the chambers of $\mathcal{C}(Q)$ (flags of $Q$). More explicitly, if $C'$ is a chamber in the same orbit as $C$, that is, $o(C')=o(C)$, we replace $C'$ by an affine copy of the complex $\mathcal{R}^{L_{o(C)}}$ that we used for $C$, such that, for each $i=0,\ldots,d-1$, the vertex $u_i$ of $D$ is mapped onto the vertex of $C'$ labelled $i$ in $\mathcal{C}(Q)$. In short, with respect to insertion of diagrams we treat $C$ and $C'$ in the same manner, and we can do so without destroying the action of $\Gamma$ because of the existence of label preserving transfer maps from $\Gamma$ between chambers in the same orbit under $\Gamma$. The resulting $(d-1)$-dimensional complex $\mathcal{C}'$ is a refinement of $\mathcal{C}(Q)$ and has the full $(d-2)$-skeleton of $\mathcal{C}(Q)$ as a subcomplex, unrefined. In particular, $\mathcal{C}'$ tiles the boundary $\partial Q$ of $Q$ and hence is topologically a $(d-1)$-sphere. By construction, $\Gamma$ acts on $\mathcal{C}'$ as a group of automorphisms.  

Clearly we may project the complex $\mathcal{C}'$ radially onto any sphere about the centroid of $Q$, and rescale the sphere (if need be) to obtain an isomorphic complex $\mathcal{T}$ which tiles the unit sphere $\mathbb{S}^{d-1}$ in a face-to-face manner by spherical convex polytopes. 

Finally, by adjoining suitable improper faces (of ranks $-1$ and $d$) to $\mathcal{C}'$ we arrive at a spherical abstract $d$-polytope, denoted $\mathcal{P}$. Then the properties of $\mathcal{P}$ described in parts (b) and (c) of the theorem are clear by construction. It remains to establish parts (a) and (d). 
\medskip

{\sf Part~(a).}\  For the proof of part (a), a more subtle choice of the parameters $m_0,\ldots,m_{d-1}$ is needed to guarantee that the polytope $\mathcal{P}$ has the property that $\Gamma(\mathcal{P})=\Gamma$. Suppose $Q$ and $\mathcal{C}(Q)$ are as before.  For a vertex $u$ of $\mathcal{C}(Q)$, we let $s_u$ denote the number of chambers containing $u$, and note that this is just the number of flags of $Q$ containing the face of $Q$ represented by~$u$. If $x$ is a vertex of any complex $\mathcal{S}$, we also write ${\rm val}_\mathcal{S}(x)$ for the valency of $x$ in the edge graph ($1$-skeleton) of $\mathcal{S}$.

It is straightforward to compute the valencies of the vertices of $\mathcal{P}$ (or $\mathcal{C}'$). The valencies of the vertices of $\mathcal{P}$ in $\mathcal{C}(Q)$ depend on $m_0,\ldots,m_{d-1}$, while those of the vertices of $\mathcal{P}$ outside of $\mathcal{C}(Q)$ do not depend on $m_0,\ldots,m_{d-1}$ but are bounded by a constant depending on $d$ and the polytopes $L_{o(C)}$. The details are as follows. For each $i=0,\ldots,d-1$, each vertex $x$ of $\mathcal{C}(Q)$ labelled $i$ is the vertex labelled $i$ in every chamber that contains it, and therefore
\begin{equation}
\label{valxi}
{\rm val}_\mathcal{P}(x) = {\rm val}_{\mathcal{C}(Q)}(x) + s_xm_{i}.
\end{equation}
If $x$ is a vertex of the central simplex in the complex $\mathcal{R}^{L_{o(C)}}$ inserted into a chamber $C$, then
\[ {\rm val}_\mathcal{P}(x) = 2(d-1) +({\rm val}_{L_{o(C)}}(x) - (d-1)) = {\rm val}_{L_{o(C)}}(x) + d-1 .\]
If $x$ is a vertex of the copy of a polytope $R_i$ inside a chamber $C$ that is not a vertex of $C$ or of the central simplex inside $C$, then ${\rm val}_\mathcal{P}(x)=d$. Finally, if $x$ is a vertex of the copy of $L_{o(C)}$ in a chamber $C$ that is not a vertex of the central simplex in $C$, then
${\rm val}_\mathcal{P}(x)={\rm val}_{L_{o(C}}(x)$,  In particular, 
there exists a constant $m$ (depending on $d$ and our choice of polytopes $L_{o(C)}$) such that 
\begin{equation}
\label{emm}
{\rm val}_\mathcal{P}(x) \leq m 
\end{equation}
for all vertices $x$ of $\mathcal{P}$ outside of $\mathcal{C}(Q)$. 

The parameters $m_i$ are chosen inductively for $i=d-1,d-2,\ldots,0$, beginning with $m_{d-1}:=m$ where $m$ is a fixed constant as in (\ref{emm}). Suppose for a moment that a specific parameter value $m_i$ has been chosen and then substituted on the left side of equation (\ref{valxi}) to give certain integers, ${\rm val}_{\mathcal{C}(Q)}(x) + s_xm_{i}$, representing vertex valencies in \mc P. In this situation we write $a_{i}$ and $b_{i}$ for the minimum or maximum of these integers ${\rm val}_{\mathcal{C}(Q)}(x) + s_xm_{i}$, respectively, taken over all vertices $x$ in $\mathcal{C}(Q)$ labelled $i$, as given in (\ref{valxi}). Thus $a_{i}\leq {\rm val}_{\mathcal{C}(Q)}(x) + s_xm_{i} \leq b_{i}$ for each vertex $x$ of $\mathcal{C}(Q)$ labelled $i$. In particular, we trivially have $m< a_{d-1}\leq b_{d-1}$. 

Proceeding inductively, we next choose $m_{d-2}$ in such a way that $b_{d-1}<a_{d-2}$. More generally, if $j\leq d-1$ and $m_{j}$ has already been chosen, we pick $m_{j-1}$ in such a way that $b_{j}<a_{j-1}$. At the final step when $j=1$, we are choosing $m_0$. Our choice of $m_0,\ldots,m_{d-1}$ then guarantees that
\begin{equation}
\label{ab}
m< a_{d-1}\leq b_{d-1} < a_{d-2}\leq b_{d-2} <\ldots\ldots < a_{1}\leq b_{1}<a_{0}\leq b_{0}.
\end{equation}
Now set $M_{i}:=[a_{i},\,b_{i}]$ for each $i$, and observe that $M_0,\ldots,M_{d-1}$ are mutually disjoint intervals. 

We now are ready to prove part (a) of the theorem. We show that if the parameters $m_0,\ldots,m_{d-1}$ are chosen in such a way that (\ref{ab}) is satisfied, then $\Gamma(\mathcal{P})=\Gamma$. Suppose $m_0,\ldots,m_{d-1}$ are chosen  such that (\ref{ab}) holds.

For the proof of part (a) we can mostly proceed as in \cite{SWil}, specifically Lemma~2. By construction, $\Gamma$ is a subgroup of $\Gamma(\mathcal{P})$, so we only need to prove the opposite inclusion. The initial steps of the proof are the same (almost word for word) as those in \cite[pp. 451-452]{SWil}. To make the present paper reasonably self-contained we reproduce here some of the arguments. 

The first step is to show that every automorphism of $\mathcal{P}$ is induced by an automorphism of~$Q$. Suppose $\gamma$ is an automorphism of $\mathcal{P}$. We want to show that $\gamma$ lies in $\Gamma$. The vertices of~$\mathcal{P}$ corresponding to vertices of $\mathcal{C}(Q)$ have higher valency than other vertices of $\mathcal{P}$ and hence must be permuted among each other by $\gamma$. Thus $\gamma$ maps vertices of $\mathcal{C}(Q)$ to vertices of $\mathcal{C}(Q)$. Moreover, by our choice of $m_0,\ldots,m_{d-1}$, the valency of each vertex of $\mathcal{C}(Q)$ labelled $i$ lies in~$M_i$ for each $i$, and the sets $M_0,\ldots,M_{d-1}$ are mutually disjoint. Hence $\gamma$ must map vertices of $\mathcal{C}(Q)$ labelled $i$ to vertices of $\mathcal{C}(Q)$ labelled $i$, for each $i$.
In particular, since the vertices of $\mathcal{C}(Q)$ labelled $0$ are precisely the vertices of $Q$, the vertices of $Q$ must be permuted by $\gamma$. Since the full $(d-2)$-skeleton of $\mathcal{C}(Q)$ is an (unrefined) subcomplex of the $(d-1)$-dimensional complex $\mathcal{C}'$ and already contains all the information about~$\mathcal{C}(Q)$ (only the chambers need to be added to the $(d-2)$-skeleton to obtain $\mathcal{C}(Q)$), it then follows that $\gamma$ induces a label preserving automorphism of~$\mathcal{C}(Q)$ mapping vertices of $Q$ to vertices of $Q$. Here it helps to bear in mind that $\mathcal{C}(Q)$ lies on a sphere.

We show that $\gamma$ induces an automorphism $\gamma_Q$ (say) of $Q$ itself, and that $\gamma_Q$ determines $\gamma$ uniquely. Since every face of the polytope $Q$ is uniquely determined by the set of flags of $Q$ containing this face, it is clear that every vertex of $\mathcal{C}(Q)$ is uniquely determined by the chambers of $\mathcal{C}(Q)$ containing this vertex. Now if $F$ is an $i$-face of $Q$ and $w_F$ is the corresponding vertex labelled $i$ in $\mathcal{C}(Q)$, then $\gamma(w_F)$ is also a vertex labelled $i$ in $\mathcal{C}(Q)$ and hence must corresponds to an $i$-face of $Q$. This $i$-face is simply $\gamma(F)$. Note here that $\gamma$ induces an isomorphism between the vertex-stars of $w_F$ and $\gamma(w_F)$ in $\mathcal{C}(Q)$; in particular, chambers of $\mathcal{C}(Q)$ containing $w_F$ are mapped in a one-to-one and label preserving manner to chambers containing $\gamma(w_F)$. 

It remains to show that $\gamma_Q$ determines $\gamma$ uniquely. To this end suppose $\gamma_Q$ is the identity map on $Q$. Then the automorphism induced by $\gamma$ on $\mathcal{C}(Q)$, $\gamma_{\mathcal{C}(Q)}$ (say), is also the identity map on $\mathcal{C}(Q)$, since the simplices in $\mathcal{C}(Q)$ just represent the chains of the boundary complex of $Q$, such that vertices of $\mathcal{C}(Q)$ labelled $i$ correspond to faces of $Q$ of rank~$i$. With regards to chamber replacement in $\mathcal{C}(Q)$ by complexes like $\mathcal{R}^{L_{o(C)}}$, note that $\gamma$ maps a complex like $\mathcal{R}^{L_{o(C)}}$ placed into a chamber, to a similar such complex placed into the image chamber under~$\gamma$. But since $\gamma$ fixes every face of a chamber of $\mathcal{C}(Q)$, which in a complex like $\mathcal{R}^{L_{o(C)}}$ becomes the outer simplex, $\gamma$ then must also fix the entire complex inserted into the chamber. This follows from a simple connectedness argument. The outer simplex of a complex $\mathcal{R}^{L_{o(C)}}$ can be joined to every tile in $\mathcal{R}^{L_{o(C)}}$ by a finite sequence of successively adjacent tiles (successive tiles meet in a facet). Beginning with the outer simplex on which $\gamma$ is the identity map, we then can move along the sequence to show that $\gamma$ is also the identity map on every tile in the sequence. Hence $\gamma$ is the identity map on the entire complex $\mathcal{C}'$ and therefore also on $\mathcal{P}$. 

Thus $\Gamma(\mathcal{P})$ can be viewed as a subgroup of $\Gamma(Q)$ containing $\Gamma$. The final step consists of showing that $\Gamma(\mathcal{P})=\Gamma$. Here the arguments of \cite[pp. 452-453]{SWil} need to be modified as follows.

Suppose that $\Gamma$ is a proper subgroup of $\Gamma(\mathcal{P})$. Then since $\Gamma(Q)$ acts freely on the chambers of $\mathcal{C}(Q)$, and $\Gamma(\mathcal{P})$ is a subgroup of $\Gamma(Q)$, the orbits of chambers $C$ of $\mathcal{C}(Q)$ under $\Gamma(\mathcal{P})$ are strictly larger than those under $\Gamma$. In particular, there are two different orbits $o(C_1)$ and $o(C_2)$ of chambers $C_1$ and $C_2$  under $\Gamma$, which lie in the same orbit under $\Gamma(\mathcal{P})$. Any automorphism $\gamma$ of $\Gamma(\mathcal{P})$ which maps a chamber $C_{1}'$ in $o(C_1)$ to a chamber $C_{2}'$ in $o(C_2)$ induces an isomorphism between the corresponding complexes $\mathcal{R}^{L_{o(C_1)}}$ and $\mathcal{R}^{L_{o(C_2)}}$ inserted into $C_{1}'$ and $C_{2}'$, respectively. However, this is impossible, since the complexes $\mathcal{R}^{L_{o(C)}}$ are mutually non-isomorphic, by our choice of the polytopes $L_{o(C)}$. Thus $\Gamma(\mathcal{P})=\Gamma$. This completes the proof of part (a) of the theorem.
\medskip

{\sf Part~(d).}\  For the proof of part (d) we must further refine our arguments. So let $\Gamma$ be a subgroup of the geometric symmetry group $G(Q)$ of $Q$. In this case we choose the standard barycentric subdivision for $\mathcal{C}(Q)$ (with the vertices of $\mathcal{C}(Q)$ at the centroids of the faces of $Q$). Then $\mathcal{C}(Q)$ is invariant under $\Gamma$, since geometric symmetries of convex polytopes map face centroids to face centroids. Next we proceed as before and replace, for each chamber orbit~$o(C)$ under $\Gamma$, one of its chambers, $C$ (say), by an affine copy of the corresponding complex~$\mathcal{R}^{L_{o(C)}}$ such that, for each $i=0,\ldots,d-1$, the vertex $u_i$ of $D$ is mapped onto the vertex of $C$ labelled $i$. For the chamber replacement of the remaining chambers of $\mathcal{C}(Q)$ we use as transfer maps the elements of $\Gamma$, which now are geometric symmetries of $\mathcal{C}(Q)$. More explicitly, if $C'$ is a chamber of $\mathcal{C}(Q)$ with $o(C')=o(C)$, and $\gamma$ is the (unique, labeling preserving) symmetry that maps $C$ to $C'$, we replace $C'$ by the image of $\mathcal{R}^{L_{o(C)}}$ under~$\gamma$. Then the overall structure is also invariant under $\Gamma$, and the same holds for its (scaled) projected image $\mathcal{T}$ on $\mathbb{S}^{d-1}$. Note that $\mathcal{T}$ cannot acquire geometric symmetries which do not belong to $\Gamma$, since these would also give combinatorial symmetries, which is impossible by part (a). This completes the proof of (d).
\end{proof}

Our next theorem is based on Theorem~\ref{thm1} and deals with geometric symmetry breaking results for convex polytopes.

\begin{theorem}
\label{thm2}
Let $d\geq 3$, let $Q$ be a convex $d$-polytope, and let $\Gamma$ be a subgroup of $\Gamma(Q)$.
The abstract polytope $\mathcal{P}$ of Theorem~\ref{thm1} may be realized by a convex $d$-polytope $P$. Moreover, if~$\Gamma$ is a subgroup of $G(Q)$, then $P$ can be chosen in such a way that $G(P)=\Gamma=\Gamma(P)$.
\end{theorem} 

\begin{proof}
The proof of the first statement is the same as the proof of \cite[Theorem 4.2]{SWil}:\ first the complex $\mathcal{C}(Q)$ is realized by a convex $d$-polytope $R$, and then all subsequent modifications to the boundary of $R$ required for the construction of $\mathcal{P}$ are achieved by gluing projective copies of convex polytopes to the facets of $R$ that are sufficiently thin in the direction of the outward facing normal to the facet. The result is a convex $d$-polytope $P$.

The proof of the second statement is similar. First observe that $R$ can be chosen in such a way that $\Gamma$ lies in $G(R)$. In fact, the construction of $R$ described in the proof of \cite[Lemma 3]{SWil} respects symmetries and leads to a convex $d$-polytope $R$ whose symmetry group contains $\Gamma$ as a subgroup.  The chamber replacement can again be realized by gluing thin projective copies of convex polytopes to facets of $R$. Now this is done in two steps. First, we only glue copies to the facets of $R$ which correspond to chambers in a system of representatives for the chamber orbits $o(C)$ on $\mathcal{C}(Q)$ under $\Gamma$. Second, we use the symmetries in $\Gamma$ to attach copies to the remaining facets of $R$, such that facets of $R$ equivalent under $\Gamma$ receive projective copies which are also equivalent under $\Gamma$. Bear in mind that the boundary complex of $R$ has the structure of a labeled simplicial complex on which $\Gamma$ acts freely in a label preserving manner. If the projective copies used in the first step are sufficiently thin, then the resulting structure is a convex $d$-polytope. By construction this polytope is invariant under $\Gamma$. 
\end{proof}

Parts of Theorem~\ref{thm1} hold more generally for finite abstract polytopes. With a very similar proof we can establish
the following theorem.

\begin{theorem}
\label{thm2a}
Let $d\geq 3$, let $\mathcal{Q}$ be a finite abstract $d$-polytope, and let $\Gamma$ be a subgroup of~$\Gamma(\mathcal{Q})$. Then there exists a finite abstract $d$-polytope $\mathcal{P}$ with the following properties:\\
(a)\ $\Gamma(\mathcal{P})=\Gamma$.\\
(b)\ $\mathcal{P}$ is isomorphic to a face-to-face tessellation on the topological space 
$|\mathcal{C}(\mathcal{Q})|$ of the order complex~$\mathcal{C}(\mathcal{Q})$ of $\mathcal{Q}$ by topological copies of convex polytopes. \\
(c)\ ${\rm skel}_{d-2}(\mathcal{C}(Q))$ is a subcomplex of ${\rm skel}_{d-2}(\mathcal{P})$.
\end{theorem} 

\section{Prescribing involutions as central symmetries}\label{section-involutions}

As an application of Theorem~\ref{thm2} we consider the following problem.  Given a finite group $\Gamma$ and a subgroup $\Lambda$ of $\Gamma$, can we find a convex polytope $P$ such that
\begin{itemize}
	\item $\Gamma(P)=\Gamma$ and
	\item $\Lambda$ acts on $P$ in a predetermined way? 
\end{itemize}

We are particularly interested in the case where $\Lambda=C_2$ and $\Lambda$ is generated by a central involution $\sigma$ of $\Gamma$. We wish to find a polytope $P$ such that $\sigma$ acts on $P$ as a central symmetry; that is, abusing notation, $\sigma(x) = - x$ for all $x \in P$. Thus $P$ would be centrally symmetric under the central symmetry $\sigma$. A positive answer would give a centrally symmetric version of the results of~\cite{SWil}.  Here we show that the answer is always positive for finite abelian groups containing an involution, that is, for abelian groups of even order.

\begin{theorem}\label{thm3}
Let $\Gamma$ be a finite abelian group of even order, and let $\sigma$ be an involution of $\Gamma$. Then there is a positive integer $d$ and a centrally symmetric convex $d$-polytope $P$ in $\mathbb{E}^d$, such that $G(P)=\Gamma(P)=\Gamma$ and $\sigma$ is realized as the central symmetry of $P$, that is, $\sigma(x) = -x$ for all $x \in \mathbb{E}^d$.
\end{theorem} 

\begin{proof}
Let us begin with the case where $\Gamma$ is a cyclic group of even order with generator~$\gamma$.  Thus $\Gamma = C_{2m}$ for some $m\geq 1$, and $\sigma = \gamma^m$.  We show that there exists a polytope of the desired kind in dimension $d=4$. Consider the action of $\Gamma$ as a group of isometries on $\mathbb{E}^4$, here viewed as complex $2$-space $\mathbb{C}^2$ (with $x\in\mathbb{E}^4$ corresponding to $(u,v)\in\mathbb{C}^2$), defined by letting $\gamma$ act as the mapping 
\[(u,v) \mapsto (e^{\pi i / m}\,u, e^{\pi i / m}\,v).\] 
Notice that for all $x\in \mathbb{E}^4$,\ $||\gamma (x)||=||x||$. If we take a large enough finite set of points $S$ in $\mathbb{S}^3$ (a five-element subset $S$ in general position suffices if $m\geq 3$, although one can do with less), then the convex hull of the orbit set 
\[\Gamma\!\cdot\! S:=\{\varphi(x)\mid \varphi\in\Gamma,\ x\in S\}\] 
is a convex 4-polytope $Q$ such that $\Gamma \leq G(Q)$ and $\sigma (x) = -x$ for all $x \in Q$. We then apply the construction process underlying Theorem~\ref{thm2} to construct the desired convex $4$-polytope $P$. In other words, we get rid of all excess combinatorial symmetries outside of $\Gamma$ while preserving each element of $\Gamma$ as a geometric symmetry for $P$, including in particular the involution $\sigma$ as the central symmetry for $P$. Thus $G(P) = \Gamma=\Gamma(P)$. This settles the case when $\Gamma$ is cyclic. (Note that we cannot work with $\mathbb{E}^2$ in place of $\mathbb{E}^4$ since the corresponding statement of Theorem~\ref{thm2} fails to be true for $n=2$.)

If $\Gamma$ is abelian but not cyclic, then, by the fundamental theorem of abelian groups, we can write $\Gamma$ as a direct product of $k+1$ abelian groups $\Gamma = \Gamma_1 \times \ldots \times \Gamma_k \times \Gamma_{k+1}$ for some $k\geq 1$, so that
\begin{itemize}
	\item $\Gamma_1, \ldots, \Gamma_k$ are cyclic and of even order, and 
	\item $\sigma = (\sigma_1, \ldots, \sigma_k, 1)$, where $\sigma_i$ is an involution in $\Gamma_i$ for all $1 \le i \le k$.
\end{itemize}
The idea is to manufacture a suitable polytope for each direct factor of $\Gamma$ and then combine these polytopes into a single polytope for $\Gamma$ itself. 

We know from the above that for each direct factor $\Gamma_i$, with $1 \le i \le k$, there is a centrally symmetric 4-polytope $P_i$ in $\mathbb{E}^4$ such that $G(P_i) = \Gamma_i=\Gamma(P_i)$ and $\sigma_i$ is the central symmetry for~$P_i$. Consider the cartesian product polytope $P' := P_1 \times \ldots \times P_k$ in $\mathbb{E}^{4k}$, whose vertex set is the cartesian product of the vertex sets of the component polytopes. Clearly, $P'$ is a centrally symmetric $4k$-polytope, and the direct product $\Gamma':=\Gamma_1 \times \ldots \times \Gamma_k$ acts on $P'$ as a group of symmetries such that each factor $\Gamma_i$ acts on the ambient $4$-dimensional subspace of the component polytope $P_i$. Under this action, $(\sigma_1, \ldots, \sigma_k)$ is the central symmetry for $P'$.

For the last direct factor, $\Gamma_{k+1}$, we embed $\Gamma'':=\Gamma_{k+1}$ into a symmetric group, $S_{l+1}$ for some $l$, and then take a regular $l$-simplex $P''$ in $\mathbb{E}^l$ centered at the origin. Then $\Gamma''$ is a (generally proper) subgroup of $G(P'')=\Gamma(P'')=S_{l+1}$. Any excess  symmetries that $P''$ might have, will be trimmed at a later stage. 

We now combine these polytopes. Set $d:=4kl$. Let $V$ denote the set of points in 
$\mathbb{E}^{d}=\mathbb{E}^{4k}\otimes\mathbb{E}^{l}$ of the form $u \otimes v$, where $u$ and $v$ are vertices of $P'$ and $P''$, respectively, and $\otimes$ denotes the standard tensor product (given by $u\cdot v^T$ if $u$ and $v$ are viewed as column vectors). Then $V$ is a centrally symmetric point set, since the vertex set of $P'$ is centrally symmetric and $(-u)\otimes v = -u\otimes v$. Hence the convex hull of $V$ in $\mathbb{E}^{d}$ is a centrally symmetric convex $d$-polytope $P$. Note that for the central symmetry of $P$ it is not required that  $P''$ is centrally symmetric. 

By construction, the actions of $\Gamma'$ on $P'$ and $\Gamma''$ on $P''$ induce an action of $\Gamma=\Gamma'\times\Gamma''$ on $P$ as a group of geometric symmetries. Thus $\Gamma$ is a subgroup of $G(P)$. If we write the given involution $\sigma$ of $\Gamma$ in the form $\sigma=(\sigma',\sigma'')$ with $\sigma'=(\sigma_{1},\ldots,\sigma_{k})\in\Gamma'$ and $\sigma'':=1\in\Gamma''$, then under this action, $\sigma$ maps each vertex $u \otimes v$ of $P$ to $(-u)\otimes v = -u\otimes v$ and thus acts on $P$ as central symmetry $-\rm{id}$, as desired.

In the final step, if $P$ has any extra symmetries outside of $\Gamma$ (as will usually be the case), we can trim them down using Theorem \ref{thm2}. This finally produces the desired centrally symmetric polytope.
\end{proof}

For cyclic groups (of even order), the construction underlying Theorem~\ref{thm3} produced convex polytopes in dimension $4$. The reader might wonder if a suitable geometric representation of these groups in 3-space $\mathbb{E}^3$ can not also give a convex 3-polytope. As the following theorem shows, the answer is negative for many abelian groups. Dimension $4$ is optimal in many cases.

\begin{theorem}
\label{thm4}
If $\Gamma = C_{4m}=\langle\gamma\rangle$ for some $m\geq 1$, and $\sigma:=\gamma^{2m}$, then there is no centrally symmetric 3-polytope $P$ in $\mathbb{E}^3$ such that $\Gamma(P)=\Gamma$ and $\sigma$ is realized as the central symmetry of $P$.
\end{theorem}

\begin{proof} 
Suppose to the contrary that such a $3$-polytope $P$ exists. Then, since $\partial P$ is homeomorphic to $\mathbb{S}^2$, we can view $\Gamma$ as a group of homeomorphisms of $\mathbb{S}^2$. In particular, $\lambda := \gamma^m$ is a homeomorphism of $\mathbb{S}^2$ with $\lambda^2 = \sigma=-\operatorname{id}$ and thus its topological degree must be positive. On the other hand, the topological degree of the homeomorphism $-\operatorname{id}$ of the $k$-sphere $\mathbb{S}^k$ is $(-1)^{k+1}$, which is $-1$ when $k=2$. This leaves no possibility for the topological degree of $\lambda$. Thus $P$ cannot exist (and dimension $4$ is optimal if $\Gamma = C_{4m}$ and $\sigma:=\gamma^{2m}$).
\end{proof}

Note that if in Theorem~\ref{thm4} we had insisted on achieving $G(P)=\Gamma$ (rather than $\Gamma(P)=\Gamma$), we could have argued similarly by using the determinant of linear mappings (rather than the topological degree of homeomorphisms) to rule out the existence of $P$. In fact, the determinant of $\lambda^2$ would have to be positive, but the central inversion $-\operatorname{id}$ has determinant~$-1$ in dimension $3$.

On the other hand, for cyclic groups of the form $\Gamma = C_{2m}=\langle\gamma\rangle$ with $m$ odd, and $\sigma:=\gamma^{m}$, we can indeed find a convex 3-polytope in $\mathbb{E}^3$ such that $\Gamma(P)=\Gamma$ and $\sigma$ is realized as the central symmetry of~$P$. This can be obtained as follows. Consider a bipyramid $P'$ in $\mathbb{E}^3$ over a regular $2m$-gon in the $xy$-plane centered at the origin $o$, where the two apices lie symmetrically on the $z$-axis on different sides of the $xy$-plane. Clearly, $P'$ is invariant under the rotatory reflection $\gamma$ of order $2m$ which is the product of the rotation by $\pi/m$ about the $z$-axis and the reflection in the $xy$-plane. Thus $C_{2m}=\langle\gamma\rangle\leq G(P')$ and $\gamma^{m}=-\operatorname{id}$. Note that $G(P')$ is strictly larger than $C_{2m}$, since it also contains the reflection in the $xy$-plane but $C_{2m}$ does not. (In fact, $G(P')\cong D_{2m}\times C_2$.) Thus $P'$ itself does not have the required properties. However, a simple application of Theorem \ref{thm2} allows us to find a polytope $P$ by getting rid of the additional symmetries while preserving the action of $C_{2m}$. Alternatively, we can construct a polytope $P$ directly from $P'$ by attaching sufficiently thin pyramids to the facets of $P'$ in one facet orbit of $P'$ under $C_{2m}$.

\section{Some open problems}\label{section-openprobs}

Our previous discussion invites a number of open problems concerning the dimension of polytopes with preassigned symmetry groups or automorphism groups. Usually, given the group $\Gamma$ the interest is in finding polytopes of small dimension realizing $\Gamma$.  After the first version of this manuscript was uploaded to public repositories, independently of our work the three open questions below have been answered affirmatively \cite{CLS}.  We present the questions here since they may lead to more directions of research.

For a finite group $\Gamma$, we define the (combinatorial) {\it convex polytope dimension} of $\Gamma$, denoted $\operatorname{cpd}(\Gamma)$, as the smallest dimension $d$ for which there exists a convex $d$-polytope $P$ whose combinatorial automorphism group is $\Gamma$, that is, $\Gamma(P) = \Gamma$. Note that the results of \cite{SWil,Do1} are saying that for every finite group $\Gamma$, we have $\operatorname{cpd}(\Gamma)<\infty$. 

Similarly, the {\it geometric convex polytope dimension} of $\Gamma$, denoted $\operatorname{gcpd}(\Gamma)$, is defined to be the smallest dimension $d$ for which there is a convex $d$-polytope $P$ whose geometric symmetry group is $\Gamma$, that is, $G(P) = \Gamma$. The results of \cite{Do1} also imply $\operatorname{gcpd}(\Gamma)<\infty$. 

\begin{open}
For each $n$, is there a finite group $\Gamma_n$ such that $\operatorname{cpd}(\Gamma_n) \ge n$?
\end{open}

\begin{open}
For each $n$, is there a finite group $\Gamma_n$ such that $\operatorname{gcpd}(\Gamma_n) \ge n$?
\end{open}

\begin{open}
\label{open3}
Does Theorem \ref{thm3} hold for non-abelian groups $\Gamma$ and central involutions~$\sigma$ of $\Gamma$? In other words, given a finite group $\Gamma$ of even order and a central involution $\sigma$ of $\Gamma$, is there a centrally symmetric convex polytope $P$ with $G(P)=\Gamma(P)=\Gamma$ such that $\sigma$ is realized as the central symmetry $-\operatorname{id}$ of $P$?
\end{open}

Note that the proof of Theorem \ref{thm3} carries over to finite groups of the form $\Gamma = \Gamma_1 \times \Gamma_2$ where $\Gamma_1$ is abelian, and central involutions of $\Gamma$ of the form $\sigma = (\sigma_1, 1)$ where $\sigma_1$ is a central involution of~$\Gamma_1$.
\medskip

\noindent
{\bf Acknowledgment.} The authors would like to thank the anonymous referees for their valuable comments and suggestions that have improved the paper.

\bibliographystyle{spmpsci}      

\end{document}